\documentclass{article}
\usepackage{amsmath,amsfonts,amssymb,amsthm}
\usepackage{color}
\usepackage{hyperref}

\providecommand{\R}{\mathbb{R}}

\providecommand{\E}{\mathcal{E}}

\providecommand{\eps}{\varepsilon}

\newcommand{\bS}{\mathbb{S}}

\renewcommand{\leq}{\leqslant}
\renewcommand{\geq}{\geqslant}

\renewcommand{\div}{\operatorname{div}}

\newcommand{\dist}{\operatorname{dist}}
\newcommand{\Id}{\operatorname{Id}}

\newtheorem{Theorem}{Theorem}
\newtheorem{Definition}{Definition}

\newtheorem{Proposition}{Proposition}
\newtheorem{Lemma}{Lemma}

\begin{document}

\date{\today}
\title{On the  inviscid limit for the compressible Navier-Stokes system in an impermeable bounded domain.  }

\author{Franck Sueur\footnote{CNRS, UMR 7598, Laboratoire Jacques-Louis Lions, F-75005, Paris, France}
\footnote{UPMC Univ Paris 06, UMR 7598, Laboratoire Jacques-Louis Lions, F-75005, Paris, France}
}

\maketitle

\begin{abstract}
In this paper we investigate the issue of the inviscid limit for the compressible Navier-Stokes system in an impermeable fixed bounded domain.
 We consider two kinds of boundary conditions.

 The first one is the no-slip condition. 
 In this case we extend the famous conditional result \cite{Tosio} obtained by Kato  in the homogeneous  incompressible case.
  Kato proved  that if the energy dissipation rate of the viscous flow in a boundary layer of width proportional to the viscosity vanishes then the solutions of the incompressible Navier-Stokes equations converge to some solutions of the incompressible Euler equations in the energy space. 
  We provide here a natural extension of this result to the compressible case.

The other case  is the Navier condition which encodes that the fluid slips with some friction on the boundary.
 In this case we show that the convergence to the  Euler equations
holds  true in the energy space, as least when the friction is not too large.

 In both cases we use in a crucial way some relative energy estimates proved recently by 
Feireisl,  Ja Jin and  Novotn{\'y} in \cite{feireisl-jmfm}.
\end{abstract}

\section{Introduction}
In this paper we investigate the issue of the inviscid limit for the compressible Navier-Stokes system in a fixed bounded domain.
Formally dropping the viscous terms the system degenerates into the compressible Euler system.
Yet the rigorous justification is  intricate because of the appearance of boundary layers, particularly in the case of an impermeable boundary, the so-called characteristic case.

\par
\ \par
\noindent

A longstanding approach follows the seminal work of Prandtl which predicts, in the case of the homogeneous  incompressible Navier-Stokes system in an  impermeable bounded domain with no-slip condition, a sharp variation of the fluid velocity in a boundary strip of width proportional to the square of the viscosity factor.
Yet this approach seems to fail to justify the inviscid limit  in general, see \cite{DGVD,guo} and the references therein.

On the other hand Kato proved in  \cite{Tosio} the following conditional result: if the energy dissipation rate of the viscous flow in a boundary layer of width proportional to the viscosity vanishes then the solutions of the incompressible Navier-Stokes equations converge to some solutions of the incompressible Euler equations in the energy space. 
This width is much smaller than the one given by Prandtl's theory, what seems to indicate that one has to go beyond Prandtl's description to understand the inviscid limit. 

\par
\ \par
\noindent

As mentioned in the survey \cite{E} by E,  not much is known about the compressible case. Nevertheless let us mention the paper \cite{xy} which tackles the linearized Navier-Stokes equations for a compressible viscous fluid in the half-plane, the paper \cite{rousset} which deals with the one-dimensional case and the paper  \cite{gmwz} which treats the noncharacteristic case.

The first main result of this paper is an extension of Kato's result to the compressible case.
The second result deals with the case where a Navier condition is prescribed on the boundary. 
This condition encodes that  the fluid slips with some friction on the boundary. 
In this case we prove that if the friction is not too large with respect to the viscosity
then the solutions of the compressible Navier-Stokes equations converge to some solutions of the compressible Euler equations in the energy space in the inviscid limit. 
This extends some earlier results obtained by  \cite{BGP,Iftimie-Planas,Paddick,xin} in the incompressible case.

\par
\ \par
\noindent

Let us now say a few words about the technics employed in this paper.
In both cases we start with the observation that the results obtained  in the incompressible case  hinted above (i.e. \cite{Tosio,Iftimie-Planas,Paddick,BGP,xin}) use a strategy somehow related to the issue of weak-strong uniqueness, where one compares two solutions of the same equation, only one of which being smooth. 
In the results  \cite{Tosio,Iftimie-Planas,Paddick,BGP,xin} there is also a comparison between a weak and a strong solution but the word ``solution'' does not 
 not refer to the same equation.
Indeed they consider a smooth solution of the Euler equations and a weak solution of the Navier-Stokes equations. 
Here ``weak solution" refers to the solution constructed by Leray in \cite{LerayActa,LerayJMPA}, whereas local in time existence and uniqueness  of classical solutions of the Euler equations  is known since the work \cite{Wolibner} of   Wolibner. 
One may wish  to apply the weak formulation of the Navier-Stokes equations with the solution of the Euler equations  as a test function. 
However, in the case of the no-slip condition, the Euler solution does not satisfy all the boundary condition required to be an admissible test function.  
In order to overcome this difficulty Kato  introduced a corrector which, added to the Euler solution, provided a smooth test function which satisfies the no-slip condition and yet quite close to the  Euler solution. This corrector is referred to as a  ``fake'' layer, as there is no reason that it describes what really happens in the boundary's neighbourhood.
Then a few standard manipulations provide an estimate of the difference between the Navier-Stokes  and the Euler solutions in the energy space.

In the sequel we adapt this strategy to the compressible case. 
In this setting the existence of global weak solutions of the Navier-Stokes equations is known since the pioneering work  \cite{lions} of Lions, later improved by Feireisl, Novotn{\'y} and Petzeltov\'a in \cite{feireisl}, see also \cite{feireisl-book1,feireisl-book2}, for various boundary conditions.
On the other hand  the local  in time existence of strong solutions of the compressible Euler equations is well-known since the works \cite{agemi,bdv,ebin1,ebin2,Schochet}.
Here, we will also get  inspired by some recent breakthroughs  in the issue of weak-strong uniqueness of the  compressible  Navier-Stokes equations, in particular by using  some  relative energy estimates  recently proved by  Feireisl,  Ja Jin and  Novotn{\'y} in \cite{feireisl-jmfm}.
This result is somehow reminiscent of the pioneering papers \cite{Dafermos,DiP,Yau}. Let us also refer to the recent works \cite{Gallay,LSR,Vasseur} and the references therein.
Unlike the energy method used in the incompressible case this will provide a non-symmetric ``measure'' of the difference between the Euler and the Navier-Stokes solutions.
Yet, as in the weak formulations, the relative energy estimates  involve some test functions which satisfy some boundary conditions; and 
as in the incompressible case the no-slip condition is more intricate than the Navier condition as the test functions also have to satisfy the no-slip conditions, which are not satisfied by the Euler solution.
We will therefore adapt, in this case,   Kato's corrector construction in order to fit with the compressible setting.

\subsection{The compressible Navier-Stokes system}
In this paper we consider the compressible Navier-Stokes system:
\begin{eqnarray}
\label{NS1}
\partial_t \rho + \div (\rho u ) = 0 ,
\\ \label{NS2}
\partial_t  (\rho u ) +  \div (\rho u \otimes u) + \nabla_x  p(\rho) = \eps \div  \bS (\nabla_x  u)  ,
\end{eqnarray}
where \
\begin{equation}
\label{NS3}
 \bS (\nabla_x  u)  := \mu \Big( \nabla_x  u + (\nabla_x  u)^T ) - \frac23 ( \div u ) \Id \Big) + \eta (\div u) \Id ,
\end{equation}
in a bounded regular domain 
\begin{equation*}
\Omega \subset \R^3 .
\end{equation*}
Above, the unknowns are the fluid density $\rho (t,x)$, defined on $[ 0,+\infty) \times \Omega$ with values in $[ 0,+\infty) $ 
 and the fluid velocity $ u (t,x)$, defined on $[ 0,+\infty) \times \Omega$ with values in $\R^{3}$, whereas $\mu > 0 $ and $ \eta \geq 0 $ are two viscosity coefficients, and $\eps >0$ is a scaling factor.
 In \eqref{NS3} the notation $\nabla_x  u$ stands for the Jacobian matrix of the vector field $u$ and $(\nabla_x  u)^T$ denotes its transpose.
Let us stress that the lower bounds on  $\mu  $ and $ \eta $ entail that  the tensor product 
\begin{equation}
\label{tensor}
\bS (\nabla_x  u)  : \nabla_x  u = \sum_{1 \leq i,j \leq  3} \frac{\mu}{2} ( \partial_i u_{j} +  \partial_j u_{i} )^{2} + (\eta - \frac23 \mu) |   \div u |^{2} 
\end{equation}
is a positively definite quadratic form with respect to $( \partial_i u_{j} )_{1 \leq i,j \leq  3}$.

We assume that there exists $\gamma >  \frac32$ such that 
the pressure $p(\rho)$ is the following function of the density 
\begin{equation}
\label{sake}
p(\rho) = \rho^{\gamma}  .
\end{equation}
Observe that this class of pressure laws includes in particular the case of  a monoatomic gas for which the adiabiatic constant $\gamma $ is $  \frac53$.

We prescribe the initial conditions:
\begin{equation*}
 \rho  \vert_{t=0} = \rho_{0} ,\quad (\rho u )  \vert_{t=0}  = \rho_{0}   u_{0}.
\end{equation*}
In this setting the existence of global weak solutions is now well understood thanks to the pioneering work  \cite{lions} of  Lions, later improved by Feireisl, Novotn{\'y} and Petzeltov\'a in \cite{feireisl}, see also \cite{feireisl-book1,feireisl-book2}, for various boundary conditions. \par \
\par

In this paper we are interested in the limit  $ \eps  \rightarrow 0 $ which is quite sensitive to the boundary conditions prescribed on the boundary $\partial \Omega$. 
In the sequel we consider two kinds of  boundary conditions.

\subsection{No-slip conditions}
In this section we  prescribe on the boundary $\partial \Omega$ of the fluid domain the following no-slip condition:
\begin{equation}
\label{noslip}
u  \vert_{\partial \Omega}= 0 .
\end{equation}
Let us recall that we mean by weak solution in this case.
\begin{Definition}
\label{Wsol}
Let be given some  initial data $(\rho_{0} , u_{0} ) $ such that $\rho_{0}  \geq 0$, 
 $\rho_{0} \in L^\gamma (\Omega) $, $\rho_{0} u_0^2 \in L^1 (\Omega) $.
 Let $q_{0} := \rho_{0} u_{0} $ which is in $ L^\frac{2\gamma}{\gamma +1}  (\Omega)$ and let $T>0$.
We say that $(\rho ,u )$ is a finite energy weak solution of the compressible Navier-Stokes system on $[0,T]$ associated to the initial data $(\rho_{0} , u_{0} ) $
if 
\begin{itemize}
\item $\rho  \in C_w ([0,T]; L^\gamma  (\Omega) )$, $\rho u  \in C_w ([0,T];  L^\frac{2\gamma}{\gamma +1}  (\Omega) )$, $u   \in L^2 (0,T ; H^1_0  (\Omega) )$,  $\rho u^2 \in C_w ([0,T];  L^1 (\Omega) )$,
\item  the identity
\begin{equation*}
\int_\Omega \rho (T,\cdot) \phi (T,\cdot) dx  - \int_\Omega \rho_{0}  \phi (0,\cdot)  dx  = \int_0^T \int_\Omega (  \rho \partial_t \phi  +  \rho  u \cdot \nabla_x  \phi ) dx dt
\end{equation*}
holds for any $\phi  \in C^\infty_c ( [0,T] \times \overline{\Omega} ; \R )$,
\item  the identity
\begin{equation*}
\int_\Omega \rho (T,\cdot) u (T,\cdot)  \cdot \phi (T,\cdot)  dx  - \int_\Omega  q_{0} \cdot  \phi (0,\cdot)  dx  = \int_0^T \int_\Omega \Big(  \rho u \cdot \partial_t \phi  +  \rho  u \otimes u :  \nabla_x  \phi + p(\rho) \div \phi -  \eps \bS (\nabla_x  u) : \nabla_x  \phi   \Big) dx dt
\end{equation*}
holds for any $\phi \in C^\infty_c ( [0,T] \times {\Omega} ; \R^3)$,

\item the energy inequality:
\begin{equation}
\label{energynoslip}
\mathcal{E} ( \rho (\tau,\cdot)  ,u  (\tau,\cdot) ) 
+   \eps  \int_0^\tau \int_\Omega\bS (\nabla_x  u) : \nabla_x  u \, dx dt
 \leq   \mathcal{E} ( \rho_0  ,u_0  )
\end{equation}
holds true for almost every $\tau \in [0,T]$, where
\begin{equation}
\label{nrj}
 \mathcal{E} ( \rho ,u ) := \int_\Omega   E( \rho ,u )  dx ,  \text{ with }  E( \rho ,u ) := \frac12 \rho | u |^2 + H( \rho)  \text{ and }    H( \rho ) :=  \frac{\rho^\gamma }{\gamma - 1} .
 \end{equation}
\end{itemize}
\end{Definition}
Let us stress that the two terms in the left hand side of \eqref{energynoslip} are nonnegative. In particular it follows from \eqref{tensor} that there exists a constant $C_0 > 0$ such that  for any $u \in H^1 (\Omega)$, 
\begin{equation}
\label{coercifnoslip}
 \int_\Omega\bS (\nabla_x  u) : \nabla_x  u \, dx \geq C_0  \int_\Omega | \nabla_x u |^2 \,  dx .
\end{equation}
Let us now state the result of global existence of weak solutions hinted above.
\begin{Theorem}[\cite{lions}, \cite{feireisl}]
\label{LF}
Let be given some  initial data $(\rho_{0} , u_{0} ) $ such that $\rho_{0}  \geq 0$, 
 $\rho_{0} \in L^\gamma (\Omega) $, $\rho_{0} u_0^2 \in L^1 (\Omega) $.
 Let $T>0$.
Then there exists a finite energy weak solution of the compressible Navier-Stokes system on $[0,T]$ associated to the initial data $(\rho_{0} , u_{0} ) $.
\end{Theorem}
Let us now investigate the issue of the inviscid limit.
Formally when the  coefficient $\eps$  is set equal to zero in \eqref{NS1}-\eqref{NS3} one obtains the following compressible Euler system:
\begin{eqnarray}
\label{E1}
\partial_t \rho^{E} + \div (\rho^{E} u^{E} ) = 0 ,
\\ \label{E2}
\partial_t  (\rho^{E} u^{E} ) +  \div (\rho^{E} u^{E} \otimes u^{E}) + \nabla_x  p(\rho^{E}) = 0  ,
\end{eqnarray}
for which one prescribes the following initial condition:
\begin{equation}
 \label{E3}
 \rho^{E}  \vert_{t=0} = \rho^{E}_{0} , \quad (\rho^{E} u^{E} )  \vert_{t=0}  = \rho^{E}_{0} u^{E}_{0} .
\end{equation}
and the following boundary condition:
\begin{equation}
\label{E4}
u^{E} \cdot n  \vert_{\partial \Omega}= 0 ,
\end{equation}
where $n$ is the unit outward normal to the domain $\Omega$.
Let us stress in particular the loss of information about the tangential components from \eqref{noslip} to \eqref{E4}.
For this system the local existence of strong solutions is well-known, cf. \cite{agemi,bdv,ebin1,ebin2,Schochet}.
\begin{Theorem}[\cite{agemi,bdv,ebin1,ebin2,Schochet}] 
\label{EulerStrong}
Let be given $(\rho^{E}_{0} , u^{E}_{0} )$ be some smooth and compatible initial data with 
$0 < \inf_{ \Omega} \rho^{E}_{0}  $ and  $ \sup_{ \Omega} \rho^{E}_{0}   (x) < \infty $.
Then there exists $T>0$ and  a unique smooth solution  $(\rho^{E} , u^{E} ) $ of \eqref{E1}-\eqref{E4} such  that 
\begin{equation}
\label{nicedensity}
0 < \inf_{(0,T) \times \Omega} \rho^{E}  \text{ and } \sup_{(0,T) \times \Omega} \rho^{E}  < \infty .
\end{equation}
\end{Theorem}
Above ``compatible'' refers to some conditions satisfied by the initial data on the boundary $\partial \Omega$ which are necessary for the existence of a strong solution. 
We refer here to  \cite{rm,Schochet} for more information on this subject.
 \par \
\par

We can now state the first main result of the paper. 
Let us denote
\begin{equation}
\label{defdist}
 d_\Omega (x) := \dist (x,  \partial   \Omega)  \text{ and }
\Gamma_{\eps} := \{  x \in  \Omega / \   d_\Omega (x) <  \eps \} ,
\end{equation}
which is well defined for $\eps >0$ small enough.

\begin{Theorem}
\label{ThKato}
Let be given $c>0$.

 Let be given $T>0$ and  $(\rho^{E} , u^{E} ) $  the strong solution of the Euler equations corresponding to an initial data $(\rho^{E}_{0} , u^{E}_{0} )$  as in Theorem \ref{EulerStrong}.

For any $\eps \in (0,1)$, let $(\rho_{0} , u_{0} ) := (\rho^{\eps}_{0} , u^{\eps}_{0} ) $  be an 
 initial data such that $\rho_{0}  \geq 0$, 
 $\rho_{0} \in L^\gamma (\Omega) $, $\rho_{0} u_0^2 \in L^1 (\Omega) $, and consider $(\rho , u) := (\rho^{\eps} , u^{\eps})$ an  associated  weak solution of the compressible Navier-Stokes system on $[0,T]$ as given by Theorem \ref{LF}.

Assume that 
\begin{equation}
\label{cvInit}
\| \rho_{0}  - \rho^{E}_{0}    \|_{L^{\gamma} (\Omega)} + \int_\Omega  \rho_{0} | u_{0} - u^{E}_{0}  |^2 \, dx  \rightarrow  0  \text{ when }\eps  \rightarrow  0 .
\end{equation}

Assume moreover that 
\begin{equation}
 \label{KatoCondition}
    \eps \int_{(0,T) \times \Gamma_{c\eps} } \Big(    \frac{\rho  | u |^{2}}{d_\Omega ^{2}}  +  \frac{ \rho^{2} ( u \cdot n )^{2}  }{d_\Omega ^{2}}     +  | \nabla_x  u |^{2} \Big) dx dt \rightarrow  0 , \text{ when }\eps  \rightarrow  0 .
\end{equation}

Then 
\begin{equation}
\label{cv}
\sup_{t \in (0,T)} \Big( \| \rho - \rho^{E}   \|_{L^{\gamma} (\Omega)} + \int_\Omega  \rho | u- u^{E} |^2 \, dx \Big) (t) \rightarrow  0  \text{ when }\eps  \rightarrow  0 .
\end{equation}
\end{Theorem}
Let us stress that \eqref{cv} implies in particular that 
\begin{equation*}
\sup_{(0,T)}   \| \rho u - \rho^{E}  u^E  \|_{L^{1} (\Omega)} \rightarrow  0    \text{ when }\eps  \rightarrow  0 ,
\end{equation*}
since
\begin{equation*}
\rho u - \rho^{E}  u^E = ( \rho  - \rho^{E} )  u^E + \sqrt{\rho}  \sqrt{\rho} ( u- u^{E} ) .
\end{equation*}
Theorem \ref{ThKato} extends to the compressible case the earlier result  \cite{Tosio} obtained by  Kato  in the homogeneous  incompressible case.
Observe in particular the condition  \eqref{KatoCondition} can be simplified when the density $\rho$ is constant thanks to Hardy's inequality into the condition
\begin{equation*}
    \eps \int_{(0,T) \times \Gamma_{c\eps} }  | \nabla_x  u |^{2}  \, dx dt \rightarrow  0    \text{ when }\eps  \rightarrow  0 ,
\end{equation*}
which is the condition used by Kato in  \cite{Tosio}  in the  incompressible case.

Let us mention that Theorem \ref{ThKato} can be easily extended to the slightly more general pressure laws as described 
   in   \cite[Eq. (2.1)]{feireisl-jmfm}.  We choose here to deal with the law \eqref{sake} for the sake of clarity.
  Actually we plan to address in a forthcoming work the case of the full Navier-Stokes-Fourier system, for which Feireisl and Novotn{\'y} have recently established in   \cite{feireisl-arma}  some relative energy estimates and weak-strong uniqueness.

Furthermore there exists many variants of Kato's argument: see for instance \cite{Wang,Masmoudi,ILL}. 
In particular let us stress that the two last ones consider some settings where a Kato type analysis provides an unconditional theorem. 
Indeed  \cite{Masmoudi} considers different horizontal and vertical viscosities, going to zero with different speeds whereas \cite{ILL}  considers an obstacle whose size goes to zero with the viscosity.
We therefore hope that the analysis performed in this paper could be useful to 
extend some of these works  in the compressible case.
%
%%%%%%%%%%%%%%%%%%%%%%%%%

%
\subsection{Navier conditions}
In this section we  prescribe on the boundary $\partial \Omega$ of the fluid domain the following Navier condition:
\begin{equation}
\label{navier}
u \cdot n   = 0 \text{ and } \eps (\bS (\nabla_x  u) n)_{tan} = \beta u_{tan} \text{ on } \partial \Omega ,
\end{equation}
where $\beta \geq 0$ is the friction coefficient and $u_{tan} $ denotes the tangential component of a vector field $u: \overline{\Omega} \rightarrow \R^{3}$ on the boundary $ \partial \Omega$.
For these boundary conditions the definition of weak solutions is adapted as follows.
\begin{Definition}
\label{WsolNavier}
Let $( \rho_{0} , u_{0})$ be an 
 initial data such that $\rho_{0}  \geq 0$, 
 $\rho_{0} \in L^\gamma (\Omega) $, $\rho_{0} u_0^2 \in L^1 (\Omega) $.
 Let $q_{0} := \rho_{0} u_{0}$ and $T>0$.
We say that $(\rho ,u )$ is a finite energy weak solution of the compressible Navier-Stokes system  on $[0,T]$ associated to the initial data $(\rho_{0} , u_{0} ) $ if 
\begin{itemize}
\item $\rho \in C_w ([0,T];  L^\gamma  (\Omega) )$, $\rho u  \in C_w ([0,T];  L^\frac{2\gamma}{\gamma +1}  (\Omega) )$, $\nabla_x u   \in L^2 (0,T ;  L^2  (\Omega) )$, $\rho u^2 \in C_w ([0,T];  L^1 (\Omega) )$,
\item the identity
\begin{equation*}
\int_\Omega \rho (T,\cdot) \phi (T,\cdot) \,  dx - \int_\Omega \rho_{0}  \phi (0,\cdot) \,  dx= \int_0^T \int_\Omega (  \rho \partial_t \phi  +  \rho  u \cdot \nabla_x  \phi ) dx dt
\end{equation*}
holds for any $\phi  \in C^\infty_c ( [0,T] \times \overline{\Omega} ; \R)$,
\item the identity
\begin{eqnarray*}
&& \int_\Omega \rho (T,\cdot) u (T,\cdot)  \cdot \phi (T,\cdot) \,  dx - \int_\Omega q_{0} \cdot  \phi (0,\cdot) \,  dx =  -  \beta \int_0^T \int_{\partial \Omega} u \cdot \phi  \,  d\sigma dt 
\\ &&\quad +\int_0^T \int_\Omega \Big(  \rho u \cdot \partial_t \phi  +  \rho  u \otimes u :  \nabla_x  \phi + p(\rho) \div \phi - \eps \bS (\nabla_x  u) : \nabla_x  \phi   \Big) dx dt
\end{eqnarray*}
holds for any $\phi \in C^\infty_c ( [0,T] \times \overline{\Omega} ; \R^3)$ such that $\phi \cdot n = 0$ on $ \partial \Omega$,
\item the energy inequality:
\begin{equation*}
 \mathcal{E} ( \rho (\tau,\cdot)  ,u  (\tau,\cdot) )  +  \eps \int_0^\tau \int_\Omega \bS (\nabla_x  u) : \nabla_x   u  \, dx dt
+\beta \int_0^\tau \int_{\partial \Omega} | u |^2 \, d\sigma dt  
 \leq   \mathcal{E} ( \rho_0  ,u_0  )  .
\end{equation*}
holds for almost every $\tau \in [0,T]$.
\end{itemize}
\end{Definition}
Above the integration element $d\sigma $ refers to the surface measure on $ \partial \Omega$.
\par
\ \par
Let us now state the second main result of this paper.
\begin{Theorem}[] 
\label{ThNavier}
Let be given $T>0$ and $(\rho^{E} , u^{E} ) $  the strong solution of the Euler equations corresponding to an initial data $(\rho^{E}_{0} , u^{E}_{0} )$  as in Theorem \ref{EulerStrong}.
For any $\eps \in (0,1)$, let $(\rho_{0} , u_{0} ) := (\rho^{\eps}_{0} , u^{\eps}_{0} ) $  be an 
 initial data such that $\rho_{0}  \geq 0$, 
 $\rho_{0} \in L^\gamma (\Omega) $, $\rho_{0} u_0^2 \in L^1 (\Omega) $ and consider $(\rho , u):= (\rho^{\eps} , u^{\eps})$ a corresponding weak solution  on $[0,T]$ in the sense of Definition \ref{WsolNavier}.

Assume that $(\rho_{0} , u_{0} )$  converges to $(\rho^{E}_{0} , u^{E}_{0} )$  when $\eps $ converges to $0$ in the sense of \eqref{cvInit}.
Assume that $\beta := \beta^\eps$ converges to $0$ when $\eps $ converges to $0$.
Then, 
\begin{equation*}
\sup_{(0,T)} \Big( \| \rho - \rho^{E}   \|_{L^{\gamma} (\Omega)} + \int_\Omega  \frac12 \rho | u- u^{E} |^2 dx \Big) (t) \rightarrow  0   \text{ when }\eps  \rightarrow  0 .
\end{equation*}
\end{Theorem}
Theorem \ref{ThNavier} extends to the compressible case the earlier results \cite{Iftimie-Planas,Paddick,BGP,xin} which tackled the homogeneous incompressible case.
As in these works Theorem \ref{ThNavier} fails to tackle the case where $\beta$ is a $O(1)$ when $\eps $ converges to $0$, which seems of special interest in view of the kinetic derivation of the Navier condition performed in \cite{MasmoudiLSR,BGP}.

%%%%%%%%%%%%%%%%%%%%%%%%%

\section{Proof of Theorem \ref{ThKato}}
\label{ProofTheoremKato}

In this section we prove Theorem  \ref{ThKato}.
We will proceed in three steps. First we will recall the recent result \cite{feireisl-iumj} about some relative energy inequalities satisfied by the weak solutions of the compressible Navier-Stokes system.
These inequalities provide a non-symmetric ``measure'' of the difference  between a weak solution of the compressible Navier-Stokes system and a smooth test function  satisfying the no-slip condition.
As the smooth solutions of the Euler equations do not satisfy the tangential part of the no-slip condition, one cannot apply directly these relative energy inequalities. 
In order to overcome this difficulty we will follow a strategy used by Kato in the incompressible setting, see \cite{Tosio}, by constructing a ``fake'' layer which, added to the Euler solution, provided a smooth test function which satisfies the no-slip condition and yet quite close to the  Euler solution.

\subsection{Relative energy inequality}

Let us recall first a quite general definition of the notion of relative entropy. 
Let be given an integer $N \geq 1$ and $\mathcal{V}$ an open subset of $\R^N$.
Given a smooth function $f: \mathcal{V} \rightarrow \R$ and $v,w \in  \mathcal{V}$ the relative entropy $f(v \vert w)$  is defined by
\begin{equation*}
f(v \vert w) := f(v) - f'(w) \cdot (v-w) - f(w) .
\end{equation*}
Applying this definition to the energy defined in \eqref{nrj}
leads to the introduction of the following  relative energy $\E ( [ \rho , u] \vert [r , U ] )$ of $ (\rho , u)$ with respect to $ (r , U )$:
\begin{eqnarray*}
 \E ( [ \rho , u] \vert [r , U ] ) :=  \int_\Omega E( [ \rho , u] \vert [r , U ] ) \, dx ,
 \text{ with } E( [ \rho , u] \vert [r , U ] ) := 
   \frac12 \rho | u-U |^2 +  H( \rho \vert r )  ,
   \\  \text{ where }  H( \rho \vert r )  := \frac{ \rho^{\gamma}}{\gamma - 1} -    \frac{\gamma (\rho - r)   r^{\gamma - 1}}{\gamma - 1}   -  \frac{r^{\gamma}}{\gamma - 1} .
\end{eqnarray*}
Note that, since $p$ is strictly convex, the quantity $H( \rho \vert r )$ is nonnegative and vanishes only when $\rho = r$.
Indeed  $H( \rho \vert r )$  provides a nice control of the difference between $\rho$ and $r$, since according to \cite[Eq. (4.15)]{feireisl-iumj}, for any compact 
$K \subset (0,+\infty)$ there exists two positive constants $c_1$ and $c_2$ such that for any $\rho \geq 0$ and for any $r \in K$, 
\begin{equation}
\label{Coince}
c_1 \Big( | \rho - r |^2 1_{| \rho - r | < 1} + | \rho - r |^\gamma 1_{| \rho - r | \geq 1}\Big) \leq  H( \rho \vert r )
\leq c_2 \Big( | \rho - r |^2 1_{| \rho - r | < 1} + | \rho - r |^\gamma 1_{| \rho - r | \geq 1}\Big) .
\end{equation}
In particular, using that the domain $\Omega$ is bounded, we infer that  for any compact 
$K \subset (0,+\infty)$  there exists a constant $C >0$ such that for any functions $\rho : \Omega \rightarrow [ 0,\infty )$ and $r: \Omega \rightarrow K$,
\begin{equation}
\label{control}
C \| \rho - r  \|_{L^{\gamma } (\Omega)}^{\gamma } \leq  \Big(  \int_\Omega  H( \rho \vert r ) dx  \Big)^{\gamma} + \int_\Omega   H( \rho \vert r )  dx ,\quad   C  \int_\Omega    H( \rho \vert r ) dx  \leq  \| \rho - r  \|^{\gamma }_{L^{\gamma } (\Omega)} +   \| \rho - r  \|^{2}_{L^{\gamma } (\Omega)}   .
\end{equation}
Let us now recall the following nice recent result which is a slight rephrasing of  \cite[Th. 2.1]{feireisl-jmfm}.
\begin{Theorem}[\cite{feireisl-jmfm}] 
\label{Entropy}
Let $T>0$ and $(\rho , u)$ be a  finite energy weak solution of the compressible Navier-Stokes system on $[0,T ]$ associated to an initial data $(\rho_{0} , u_{0} )$ as in Theorem \ref{LF}.
Then, for any smooth test functions $(r,U) :  [0,T] \times \overline{\Omega} \rightarrow (0,+\infty)  \times \R^{3}$ satisfying the no-slip condition $U  \vert_{\partial \Omega}= 0 $, we have the following relative energy inequality:
\begin{equation}
\label{relative}
\E ( [ \rho , u] \vert [r , U ] )   (\tau ) 
+ \int_0^\tau \int_\Omega \eps \bS (\nabla_x  u) : \nabla_x   u \, dx dt
 \leq  \E_{0}  + \int_0^\tau \mathcal{R}  ( \rho , u , r , U ) dt ,
  \end{equation}
for almost every $\tau \in (0,T)$, where
\begin{equation}
\label{IE}
\E_0 :=  \E ( [ \rho_0 , u_0 ] \vert [r(0,\cdot)  , U(0,\cdot) ] )  ,
 \end{equation}
and
\begin{eqnarray}
\label{reste}
\mathcal{R}  ( \rho , u , r , U ) :=
\int_\Omega  \rho \Big(  \partial_t U  +   (u \cdot \nabla_x ) U \Big) \cdot (U-u) dx 
+ \int_\Omega  \eps \bS (\nabla_x  u) : \nabla_x  U dx 
\\ \nonumber \quad +  \int_\Omega  \Big( (r- \rho) \partial_t H' (r) +  \nabla_x   H' (r)\cdot (rU-\rho u)    \Big) dx
- \int_\Omega   (\div U)  \Big( p(\rho) - p(r)   \Big) dx .
 \end{eqnarray}
\end{Theorem}
Let us stress that Theorem \ref{Entropy} can be thought as a counterpart of 
\cite[Proposition 1]{lellis} which establishes that finite energy weak solutions of the incompressible Euler equations are dissipative in the sense of Lions.

\subsection{A Kato type ``fake'' layer}
\label{Fake}
The goal of this section is to prove the following result, where we make use of the Landau notations $o(1)$ and $O(1)$ for quantities respectively converging to $0$ and bounded with respect to the limit $\eps \rightarrow 0^{+}$.
\begin{Proposition}
\label{PropositionFake}
Under the assumptions of Theorem  \ref{ThKato}
there exists  $v_{F} :=v_{F}^{\eps}   \in  C ( [0,T ]  \times \overline{\Omega} ;  \R^{3} )$, supported in  $ \Gamma_{c\eps} $, such that 
\begin{eqnarray}
 \label{Fake0}
v_{F} = O( 1 ) \text{ in }C ( [0,T ]  \times \overline{\Omega}   ),
\\  \label{Fake7}
 u^E  - v_{F} = 0  \text{ on } \partial \Omega ,
\\ \label{Fake1}
v_{F} = O( \eps^{\frac{1}{p}}) \text{ in }C ( [0,T ]  ;  L^{p} (\Omega) ), \text{ for } 1 \leq p < +\infty ,
\\ \label{Fake2}
\partial_{t} v_{F} = O( \eps^{\frac{1}{p}}) \text{ in }C ( [0,T ]  ;  L^{p} (\Omega)), \text{ for } 1 \leq p < +\infty ,
\\   \label{Fake4}
\| \nabla_x  v_{F} \|_{ L^{\infty} ([0,T ] ;   L^{2} ( \Gamma_{c\eps} )) } = O( \eps^{-\frac{1}{2}}  ) ,
\\   \label{FakeImp1}
d_\Omega \,   \nabla_x  v_{F} = O( \eps^{\frac{1}{2}} ) \text{ in }   L^{\infty} ( [0,T ] ;   L^{2} (   \Omega )),
\\   \label{FakeImp2}
d_\Omega ^{2} \, \nabla_x  v_{F} = O( \eps ) \text{ in }   C( [0,T ]  \times \overline{\Omega} ),
 \\ \label{Fake-2}
\div v_{F} =  O( 1 )  \text{ in } C ( [0,T ]  \times \overline{\Omega} ) ,
\\ \label{Fake-1}
\div v_{F} =  O( \eps^{\frac{1}{p}} )  \text{ in } C ( [0,T ]  ;  L^{p} (\Omega)),  \text{ for } 1 \leq p < +\infty .
\end{eqnarray}
\end{Proposition}
Let us recall that $d_\Omega $ and   $ \Gamma_{c\eps} $ are defined in \eqref{defdist}.
\begin{proof}
Let  $\xi: [ 0,+\infty) \rightarrow [ 0,+\infty) $ be a smooth cut-off function such that $\xi(0) = 1$ and  $\xi(r) = 0$ for $r \geq 1$. 
We define 
\begin{equation}
\label{forme}
z(x) := \xi ( \frac{ d_\Omega (x)}{c \eps} ), \quad 
 \tilde{\xi}(r) := r  \xi' (r), \quad   \tilde{z}(x) := \tilde{\xi} ( \frac{ d_\Omega (x)}{c \eps} ), \quad  \hat{\xi}(r) := r^{2}  \xi' (r),   \quad \hat{z}(x) := \tilde{\xi} ( \frac{ d_\Omega (x)}{c \eps} ) \text{  and  }  v_F  :=  z  u^E .
\end{equation}
We easily see that  \eqref{Fake0}-\eqref{Fake4} are satisfied.
In particular let us stress that the leading order term of $\nabla_x  v_{F}$ is given by the normal derivative
\begin{equation*}
n \cdot \nabla_x  v_{F} = z n \cdot \nabla_x   u^E  + \frac{1}{c \eps}  \xi' ( \frac{ d_\Omega (x)}{c \eps} )  u^E .
\end{equation*}
Thus the leading order term of $d_\Omega  \,  \nabla_x  v_{F}$ and $d_\Omega ^{2}  \, \nabla_x  v_{F}$ are given by the respective contributions of the second term in the right hand side above, which can be recast as 
$\tilde{z}  u^E $ and $ c \eps \hat{z}  u^E $,
from which we infer \eqref{FakeImp1} and \eqref{FakeImp2}.

Finally let us introduce the function 
$\phi := \frac{ u^E \cdot n }{ d_\Omega }$ which is smooth up to the boundary because of \eqref{E4}.
Then we observe that 
\begin{equation*}
 \div v_{F}= z \div u^E + u^E \cdot \nabla_x  z  =  z \div u^E + \phi  \tilde{z} ,
\end{equation*}
which yields  \eqref{Fake-2} and \eqref{Fake-1}.
\end{proof}

\subsection{Core of the proof of Theorem \ref{ThKato}}

 First using \eqref{energynoslip},  \eqref{coercifnoslip} and 
 \eqref{cvInit} we easily obtain 
that the norms 
\begin{equation}
\label{ap}
\| \rho  \|_{L^\infty  (0,T; L^\gamma  (\Omega) )} +  \| \rho u^{2} \|_{L^\infty  (0,T; L^1  (\Omega) )} + \sqrt{\eps}   \| \nabla_x u  \|_{L^2 ( (0,T) \times \Omega )}  = O(1) ,
\end{equation}
with respect to $\eps \rightarrow 0$.

Let us also say  immediately that we will make use several times of \eqref{Coince} and \eqref{control}
 with the compact 
$K := [\inf_{(0,T) \times \Omega} \rho^{E}  , \sup_{(0,T) \times \Omega} \rho^{E}  ]$
(see \eqref{nicedensity}).

Now the basic idea is to apply \eqref{relative} to 
$(r , U) =  (\rho^{E}, u^{E}   - v_{F} ).$
The following lemma isolates what exactly we expect from this idea.
\begin{Lemma}
\label{lemma}
If almost everywhere  in $ [0,T]$, there holds 
\begin{equation}
\label{crucial}
|   \mathcal{R}  ( \rho , u , \rho^{E} , U ) |  \leq  \E ( [ \rho , u] \vert [ \rho^{E}  , u^{E} ] )  + o(1) , 
\end{equation}
where the $o(1) $ is with respect to $L^{1} (0,T)$, then Theorem \ref{ThKato} holds true.
\end{Lemma}
\begin{proof}
Let us first observe that 
\begin{equation*}
\frac{1}{2} \E ( [ \rho , u] \vert [ \rho^{E}  , u^{E} ] ) \leq  \E ( [ \rho , u] \vert [ r  , U ] ) + \frac{1}{2} \int_\Omega  \rho   | v_{F}   |^{2}  dx 
=   \E ( [ \rho , u] \vert [ r  , U ] ) + o(1)
\end{equation*}
in  $L^{1} (0,T)$, thanks to Holder's inequality, \eqref{ap} and  \eqref{Fake1}.

Now, using \eqref{cvInit} and  the second inequality in \eqref{control} we obtain that 
\begin{equation*}
\E ( [ \rho_{0} , u_{0}] \vert [ \rho^{E}_{0}  , u^{E}_{0} ] ) \rightarrow  0   \text{ when }\eps  \rightarrow  0 .
\end{equation*}

Then it is sufficient to apply the Gronwall lemma to \eqref{relative} applied to $(r , U) =  (\rho^{E}, u^{E}  - v_{F} )$ taking into account \eqref{crucial} to conclude that 
\begin{equation*}
\sup_{t \in (0,T)} \Big(     \E ( [ \rho , u] \vert [ \rho^{E}  , u^{E} ] )   \Big) (t) \rightarrow  0   \text{ when }\eps  \rightarrow  0 .
\end{equation*}
Finally it remains to use the first inequality in \eqref{control} to complete the proof of Lemma \eqref{lemma}.
\end{proof}
Let us point out that the proof of Lemma \ref{lemma} above makes no use of the second term of the left hand side of \eqref{relative}.

\par
\ \par
We now prove an estimate of the form \eqref{crucial}.
We first decompose the first term of the right hand side in \eqref{reste} to get 
\begin{eqnarray*}
\mathcal{R}  ( \rho , u , \rho^{E} , U ) =
\int_\Omega  \rho \Big(  \partial_t u^{E}  +   (u^{E} \cdot \nabla_x ) u^{E} \Big) \cdot (U-u) dx 
+ \int_\Omega  \rho  \Big(   ((u - u^{E} )  \cdot \nabla_x ) u^{E} \Big)  \cdot (U-u)  dx 
\\ \nonumber \quad - \int_\Omega  \rho \Big(  \partial_t  v_{F} +   (u \cdot \nabla_x ) v_{F} \Big) \cdot (U-u)  \,   dx 
 + \int_\Omega \eps \bS (\nabla_x  u) : \nabla_x  U  \,   dx 
\\ \nonumber \quad 
+  \int_\Omega  \Big( (\rho^{E}- \rho) \partial_t H' (\rho^{E}) +  \nabla_x   H' (\rho^{E})\cdot (\rho^{E}U-\rho u)    \Big) dx
- \int_\Omega   (\div U)  \Big( p(\rho) - p(\rho^{E})   \Big) dx .
 \end{eqnarray*}

Then we deduce from \eqref{E1}-\eqref{E2} that 
\begin{equation}
\label{eulerrecast}
\partial_t  u^{E}  +  (u^{E}\cdot  \nabla_x ) u^{E} =   - \nabla_x   H' (\rho^E ) ,
\end{equation}
so that 
\begin{eqnarray*}
\mathcal{R}  ( \rho , u , \rho^{E} , U ) =
 \int_\Omega  \rho  \Big(   ((u - u^{E} )  \cdot \nabla_x ) u^{E} \Big)  \cdot (U-u)  dx 
 - \int_\Omega  \rho \Big(  \partial_t  v_{F} +   (u \cdot \nabla_x ) v_{F} \Big) \cdot (U-u) dx 
\\ \nonumber \quad  + \int_\Omega \eps \bS (\nabla_x  u) : \nabla_x  U dx 
+  \int_\Omega   (\rho^{E}- \rho) \Big( \partial_t H' (\rho^{E}) +  U \cdot \nabla_x   H' (\rho^{E})   \Big) dx 
- \int_\Omega   (\div U)  \Big( p(\rho) - p(\rho^{E})   \Big) dx .
 \end{eqnarray*}
Let us decompose the two last terms into
\begin{eqnarray*}
 \int_\Omega   (\rho^{E}- \rho) \Big( \partial_t H' (\rho^{E}) +  U \cdot \nabla_x   H' (\rho^{E})   \Big) dx 
- \int_\Omega   (\div U)  \Big( p(\rho) - p(\rho^{E})   \Big) dx
=
 \\ \quad \int_\Omega   (\rho^{E}- \rho) \Big( \partial_t H' (\rho^{E}) + u^{E}  \cdot \nabla_x   H' (\rho^{E})   \Big) dx 
 - \int_\Omega    (\div u^{E})  \Big( p(\rho) - p(\rho^{E})   \Big) dx 
\\- \int_\Omega   (\rho^{E}- \rho) \Big(  v_{F} \cdot \nabla_x   H' (\rho^{E})   \Big) dx 
+ \int_\Omega   (\div v_{F})  \Big( p(\rho) - p(\rho^{E})   \Big) dx .
 \end{eqnarray*}
Moreover using again  \eqref{E1} we obtain  that 
\begin{equation}
\label{Heq}
\partial_t H' (  \rho^{E}) +  u^{E} \cdot \nabla_x   H' (\rho^{E})  = - ( \div u^{E})  p'(\rho^{E}) ,
\end{equation}
so that 
\begin{eqnarray*}
\mathcal{R}  ( \rho , u , \rho^{E} , U ) =
 \int_\Omega  \rho  \Big(   ((u - u^{E} )  \cdot \nabla_x ) u^{E} \Big)  \cdot (U-u)  dx 
 - \int_\Omega  \rho \Big(  \partial_t  v_{F} +   (u \cdot \nabla_x ) v_{F} \Big) \cdot (U-u) dx 
 \\ \nonumber \quad + \int_\Omega \eps \bS (\nabla_x  u) : \nabla_x  U dx 
- \int_\Omega   (\div u^{E})  \Big( p(\rho) - p(\rho^{E})  - p'(\rho^{E}) (\rho -\rho^{E})  \Big) dx 
\\ \nonumber \quad  - \int_\Omega   (\rho^{E}- \rho) \Big(  v_{F} \cdot \nabla_x   H' (\rho^{E})   \Big) dx 
+ \int_\Omega  ( \div v_{F} ) \Big( p(\rho) - p(\rho^{E})   \Big) dx
=: \mathcal{R}_{1} + \ldots + \mathcal{R}_{6} .
 \end{eqnarray*}
We decompose $ \mathcal{R}_{1}$ into 
\begin{equation*}
 \mathcal{R}_{1} =  - \int_\Omega  \rho  ((u - u^{E} )  \otimes (u - u^{E} ) ) : \nabla_x  u^{E}  \,   dx 
 -  \int_\Omega  \rho  \Big(   ((u - u^{E} )  \cdot \nabla_x ) u^{E} \Big)  \cdot v_{F} dx .
\end{equation*}
Therefore, using the Cauchy-Schwarz inequality, we obtain that there exists a constant $C>0$ such that 
\begin{eqnarray*}
|   \mathcal{R}_{1} |  &\leq&  C  \int_\Omega  \rho   | u - u^{E}  |^{2}  dx + 
C  \Big( \int_\Omega  \rho   | u - u^{E}  |^{2}  dx \Big)^{\frac{1}{2}}  \Big( \int_\Omega  \rho   | v_{F}   |^{2}  dx \Big)^{\frac{1}{2}} ,
\\ &\leq& C  \int_\Omega  \rho   | u - u^{E}  |^{2}  dx
+ \tilde{C}  \Big( \int_\Omega  \rho   | u - u^{E}  |^{2}  dx \Big)^{\frac{1}{2}}  \eps^{\frac{\gamma}{\gamma - 1}} ,
\end{eqnarray*}
thanks to Holder's inequality, \eqref{ap} and  \eqref{Fake1}.
Then using Young's inequality, we get that  there exists a constant $C>0$ such that 
\begin{equation}
\label{com}
|   \mathcal{R}_{1} |  \leq  C  \int_\Omega  \rho   | u - u^{E}  |^{2}  \,    dx + C  \eps^{\frac{2\gamma}{\gamma - 1}} \leq  C  \E ( [ \rho , u] \vert [ \rho^{E}  , u^{E} ] )  + o(1) .
\end{equation}
We decompose $ \mathcal{R}_{2}$ into 
\begin{equation*}
 \mathcal{R}_{2} = - \int_\Omega  \rho \partial_t  v_{F} \cdot (u^{E} -u) dx
 +  \int_\Omega  \rho \partial_t  v_{F} \cdot v_{F} dx
  - \int_\Omega  \rho \Big(  (u \cdot \nabla_x ) v_{F} \Big) \cdot U dx 
  + \int_\Omega  \rho  (u \otimes u) : \nabla_x  v_{F} dx =: \mathcal{R}_{2,a} + \mathcal{R}_{2,b} + \mathcal{R}_{2,c} + \mathcal{R}_{2,d} .
\end{equation*}
Using the Cauchy-Schwarz inequality, we get 
\begin{equation*}
|    \mathcal{R}_{2,a} | \leq  \Big( \int_\Omega  \rho |  \partial_t  v_{F} |^{2}  \,   dx    \Big)^{\frac{1}{2}} \,   \Big(\int_\Omega  \rho   | u - u^{E}  |^{2}  \,   dx    \Big)^{\frac{1}{2}}  .
\end{equation*}
Then using Holder's inequality, \eqref{Fake2}, \eqref{ap} and Young's inequality, we get that $\mathcal{R}_{2,a}$ can be bounded as $\mathcal{R}_{1}$ in
\eqref{com}.
Using again Holder's inequality, \eqref{Fake1}, \eqref{Fake2} and \eqref{ap} we obtain that 
$\mathcal{R}_{2,b}  = o(1)$ in  $L^{1} (0,T)$.
Next, using \eqref{forme}, we decompose $ \mathcal{R}_{2,c}$ as 
\begin{equation*}
 \mathcal{R}_{2,c} = - \int_\Omega   \sqrt{\rho} z U \cdot \Big( \sqrt{\rho}  u  \cdot \nabla_x ) u^E \Big)    \,  dx 
 - \int_\Omega  \frac{  \rho u \cdot n}{d_\Omega }  \tilde{z} (U \cdot u^E )     . 
\end{equation*}
We bound the first term of the right hand side above thanks to  Holder's inequality,  \eqref{ap} and the uniform boundedness of $ \nabla_x u^E$ and of $U$ (see \eqref{Fake0}) by $C \|  z \|_{L^{\frac{2\gamma}{\gamma -1}}  (\Omega)}$ which is $o(1)$ in $L^{1} (0,T) $.
In order to bound  the second term we 
use again that $U$ is bounded in $C ( [0,T ]  \times \R^{3} )$, as well as the condition on the support of $ \tilde{z}$, the Cauchy-Schwarz inequality, 
\eqref{KatoCondition} and that  $\| \tilde{z}   \|_{L^{2} (\Omega)} = O(\eps^{\frac{1}{2}} )$.
Thus we obtain  $ \mathcal{R}_{2,c} =  o(1) $ in $ L^{1} (0,T) $.
Regarding the term $ \mathcal{R}_{2,d} $, we use that
\begin{equation*}
 \mathcal{R}_{2,d} = \int_\Omega  (  \frac{\sqrt{\rho}}{d_\Omega } u    \otimes  \frac{\sqrt{\rho}}{d_\Omega } u  ) : d_\Omega ^2 \nabla_x  v_{F}  \,   dx  ,
\end{equation*}
the condition on the support of $v_{F}$, the condition \eqref{KatoCondition} and the estimate \eqref{FakeImp2} to obtain    $ \mathcal{R}_{2,d} =  o(1) $ in $ L^{1} (0,T) $.

We decompose $ \mathcal{R}_{3}$ into 
\begin{equation*}
 \mathcal{R}_{3} = - \int_\Omega \eps \bS (\nabla_x  u) : \nabla_x  v_{F}  \,  dx 
 + \int_\Omega \eps \bS (\nabla_x  u) : \nabla_x  u^{E}  \,    dx 
\end{equation*}
so that, by the Cauchy-Schwarz inequality,  \eqref{Fake4}, the condition on the support of $v_{F}$, \eqref{ap} and \eqref{KatoCondition}  we obtain 
\begin{equation*}
   \mathcal{R}_{3} =  o(1) \text{ in } L^{1} (0,T) . 
\end{equation*}
Now, we observe that 
\begin{equation*}
  \mathcal{R}_{4} =
- (\gamma -1) \int_\Omega (  \div u^{E})   H( \rho \vert  \rho^{E})  dx .
\end{equation*}
Therefore, 
\begin{equation*}
|   \mathcal{R}_{4} |  \leq  C  \int_\Omega H( \rho \vert  \rho^{E}) dx  .
\end{equation*}
We decompose $ \mathcal{R}_{5}$ into 
\begin{equation*}
  \mathcal{R}_{5} =  
- \int_{\Omega \cap \{  | \rho -\rho^{E} | < 1  \}}  (\rho^{E}- \rho) \Big(  v_{F} \cdot \nabla_x   H' (\rho^{E})   \Big) dx 
- \int_{\Omega \cap \{  | \rho -\rho^{E} | \geq 1  \}}  (\rho^{E}- \rho) \Big(  v_{F} \cdot \nabla_x   H' (\rho^{E})   \Big) dx  .
\end{equation*}
Using  Holder's inequality, \eqref{Fake1} and \eqref{Coince}, we obtain that there exists a constant $c>0$ such that 
\begin{equation*}
c |   \mathcal{R}_{5} |    \leq \eps +  \int_\Omega H( \rho \vert  \rho^{E}) dx .
\end{equation*}
We decompose $ \mathcal{R}_{6}$ into 
\begin{eqnarray*}
  \mathcal{R}_{6} =  \int_\Omega (  \div v_{F})  \Big( p(\rho) - p(\rho^{E})  - p' (\rho^{E})  (\rho -\rho^{E})  \Big) dx 
  +  \int_{\Omega \cap \{  | \rho -\rho^{E} | < 1  \}}  (\div v_{F})  p' (\rho^{E}) (\rho -\rho^{E})   dx  
 \\ \quad  + \int_{\Omega \cap \{  | \rho -\rho^{E} |  \geq 1  \}} ( \div v_{F} ) p' (\rho^{E})  (\rho -\rho^{E})   dx .
\end{eqnarray*}
Thus, using Holder's inequality and \eqref{Coince}, we obtain that there exists a constant $c>0$ such that 
\begin{equation*}
c |   \mathcal{R}_{6} |    \leq 
\|   \div v_{F}  \|_{L^{\infty} (\Omega )}  \int_\Omega H( \rho \vert  \rho^{E}) dx
+ \|  \div v_{F}  \|_{L^{2} (\Omega )}   (\int_\Omega H( \rho \vert  \rho^{E}) dx)^{\frac{1}{2}}
+ \|  \div v_{F}  \|_{L^{\frac{\gamma}{\gamma - 1}} (\Omega )}  (\int_\Omega H( \rho \vert  \rho^{E}) dx)^{\frac{1}{\gamma}} .
\end{equation*}
Using now \eqref{Fake-2}, \eqref{Fake-1} and  Young's inequality, we get that  there exists a constant $c>0$ such that 
\begin{equation*}
c |   \mathcal{R}_{6} |    \leq \eps +  \int_\Omega H( \rho \vert  \rho^{E}) dx  .
\end{equation*}
Summing the previous estimates provides an estimate of the form \eqref{crucial} and therefore concludes the proof of Theorem \ref{ThKato}.

\section{Proof of Theorem \ref{ThNavier}}
\label{ProofTheoremNavier}

In this section we prove Theorem \ref{ThNavier}.
We proceed as in the previous section. 
Actually it is even a little bit more simple since we will not need any fake layer here.
The reason is that the relative energy inequality is valid, in this case, for any smooth test function whose normal velocity vanishes on the boundary. 
We will therefore apply this inequality directly to the Euler solution. 

\subsection{Relative energy inequality}

In the case of the Navier boundary conditions the   finite energy weak solutions of the compressible Navier-Stokes system enjoy the following 
relative energy inequality which is a slight rephrasing of \cite[Eq. (3.11)-(3.12)] {feireisl-jmfm}.
\begin{Theorem}[\cite{feireisl-jmfm}] 
\label{EntropyNavier}
Let $T>0$ and  $(\rho , u)$ be a  finite energy weak solution of the compressible Navier-Stokes system  associated to an initial data $(\rho_{0} , u_{0} )$  such that $\rho_{0}  \geq 0$,  $\rho_{0} \in L^\gamma (\Omega) $, $\rho_{0} u_0^2 \in L^1 (\Omega) $ in the sense of Definition \ref{WsolNavier}.
Then, for any smooth test functions $(r,U):  [0,T] \times \overline{\Omega} \rightarrow (0,+\infty) \times \R^{3}$ satisfying 
$U \cdot n  \vert_{\partial \Omega}= 0 $,
 we have the following relative energy inequality:
\begin{equation}
\label{relativeNavier}
\E ( [ \rho , u] \vert [r , U ] )   (\tau )  
+ \eps \int_0^\tau \int_\Omega  \bS (\nabla_x  u) : \nabla_x  u \, dx
+ \beta \int_0^\tau \int_{\partial \Omega} | u|^2 \, d\sigma dt  
 \leq  \E_{0}  + \int_0^\tau \tilde{\mathcal{R}}  ( \rho , u , r , U ) \, dt ,
  \end{equation}
for almost every $\tau$, where $\E_0$ is given by \eqref{IE}
and
\begin{eqnarray*}
\tilde{\mathcal{R}}  ( \rho , u , r , U ) :=
\int_\Omega  \rho \Big(  \partial_t U  +   (u \cdot \nabla_x ) U \Big) \cdot (U-u) \, dx 
+ \eps \int_\Omega  \bS (\nabla_x  u) : \nabla_x  U \, dx 
+  \beta  \int_{\partial \Omega} u \cdot U \, d\sigma   
\\ +  \int_\Omega  \Big( (r- \rho) \partial_t H' (r) +  \nabla_x   H' (r)\cdot (rU-\rho u)    \Big)\,  dx
- \int_\Omega   (\div U)  \Big( p(\rho) - p(r)   \Big) \, dx .
 \end{eqnarray*}
\end{Theorem}

\subsection{Proof of Theorem \ref{ThNavier}}

We apply \eqref{relativeNavier} to 
$(r , U) =  (\rho^{E}, u^{E}  )$ and we use \eqref{eulerrecast} and then \eqref{Heq} to get 
\begin{eqnarray*}
\tilde{\mathcal{R}}  ( \rho , u , \rho^{E} , u^{E} ) &=&
 \int_\Omega  \rho  \Big(   ((u - u^{E} )  \cdot \nabla_x ) u^{E} \Big)  \cdot (u^{E}-u)  dx 
 +  \beta  \int_{\partial \Omega} u \cdot u^{E}  d\sigma   
 + \eps \int_\Omega  \bS (\nabla_x  u) : \nabla_x  u^{E} dx 
\\ \nonumber && 
+  \int_\Omega   (\rho^{E}- \rho) \Big( \partial_t H' (\rho^{E}) +  u^{E} \cdot \nabla_x   H' (\rho^{E})   \Big) dx 
 - \int_\Omega    (\div u^{E})  \Big( p(\rho) - p(\rho^{E})   \Big) dx ,
\\ &=&
 - \int_\Omega  \rho  ((u - u^{E} )  \otimes (u - u^{E} ) ) : \nabla_x  u^{E} dx 
 +  \beta  \int_{\partial \Omega} u \cdot  u^{E}  d\sigma   
 + \eps \int_\Omega  \bS (\nabla_x  u) : \nabla_x  u^{E} dx 
\\ \nonumber && \quad
- \int_\Omega    (\div u^{E})  \Big( p(\rho) - p(\rho^{E})  - p'(\rho^{E}) (\rho -\rho^{E})  \Big) dx 
 =: \tilde{\mathcal{R}}_{1} + \ldots  \tilde{\mathcal{R}}_{4} .
 \end{eqnarray*}
We have 
\begin{equation}
\label{R1}
|   \tilde{\mathcal{R}}_{1} |  \leq  C  \int_\Omega  \rho   | u - u^{E}  |^{2}  .
\end{equation}
On the other hand, we have:
\begin{equation}
\label{R2}
|   \tilde{\mathcal{R}}_{2} |  \leq  \frac{1}{2} \beta  \int_{\partial \Omega} | u|^2 \, d\sigma 
+  \frac{1}{2} \beta  \int_{\partial \Omega} | u^{E} |^2 \, d\sigma   .
\end{equation}
Therefore the term produced in the right hand side of \eqref{relativeNavier}
by the first term of the right hand side above can be absorbed by the third term of the left hand side of \eqref{relativeNavier}, whereas the other term goes to zero as $\beta$ goes to $0$ when $\eps $ goes to $0$.

Regarding the term $ \tilde{\mathcal{R}}_{3}$, we use the Cauchy-Schwarz inequality and  the Young inequality in order to get 
\begin{equation}
\label{R3}
|   \tilde{\mathcal{R}}_{3} |  \leq  \frac{C_0 \eps}{2}  \int_\Omega | \nabla_x  u |^2 \, dx 
+  C \eps    ,
\end{equation}
where $ C_0$ is the constant appearing in \eqref{coercifnoslip}.
Therefore the term produced in the right hand side of \eqref{relativeNavier}
by the first term of the right hand side above can be absorbed by the second term of the left hand side of \eqref{relativeNavier}, whereas the other term goes to zero  when $\eps $ goes to $0$.

Finally, we have 
\begin{equation}
\label{R4}
|   \tilde{\mathcal{R}}_{4} |  \leq  C \E ( [ \rho , u] \vert [\rho^{E}, u^{E} ] )    .
\end{equation}
Therefore it is sufficient to sum the  estimates \eqref{R1}-\eqref{R4} and to use the Gronwall lemma to conclude the proof of Theorem \ref{ThNavier}.

\par
\ \par
\noindent
{\bf Acknowledgements.} The  author is partially supported by the Agence Nationale de la Recherche, Project CISIFS,  grant ANR-09-BLAN-0213-02.

\end{document}